\documentclass[journal,twoside,web]{ieeecolor}
\usepackage{generic}
\usepackage{cite}
\usepackage{amsfonts}
\usepackage{algorithmic}
\usepackage{textcomp}
\usepackage{graphics}   % for pdf, bitmapped graphics files
\usepackage{graphicx,import}   % for pdf, bitmapped graphics files
\usepackage{epsfig}		% for postscript graphics files
\usepackage{mathptmx}   % assumes new font selection scheme installed
\usepackage{times}      % assumes new font selection scheme installed
\usepackage{amsmath}    % assumes amsmath package installed
\usepackage{amssymb}    % assumes amsmath package installed
\usepackage{epstopdf}
\usepackage{color}
\usepackage[dvipsnames]{xcolor}
\usepackage{mathtools}
\usepackage{relsize}
\usepackage{enumerate}

\def\BibTeX{{\rm B\kern-.05em{\sc i\kern-.025em b}\kern-.08em
    T\kern-.1667em\lower.7ex\hbox{E}\kern-.125emX}}

\DeclareMathOperator*{\esssup}{ess\,sup}
\DeclareMathOperator*{\essinf}{ess\,inf}

\newtheorem{theorem}{Theorem}[section]
\newtheorem{lemma}[theorem]{Lemma}
\newtheorem{proposition}[theorem]{Proposition}

\newtheorem{definition}[theorem]{Definition}

\newtheorem{remark}{Remark}
\newtheorem{example}{Example}

\begin{document}
\title{A Lyapunov approach for the exponential stability of a damped Timoshenko beam}
\author{Andrea Mattioni, Yongxin Wu, Yann Le Gorrec
\thanks{Submitted for review on February 24, 2022. This work has received funding from the European Union’s Horizon 2020 research and innovation programme under the Marie Sklodowska- Curie grant agreement No 765579. This work has been supported by the EIPHI Graduate School (contract ``ANR-17-EURE-0002"), by the ANR IMPACTS project (contract ``ANR-21-CE48-0018") and the MIAI@Grenoble Alpes (contract ``ANR-19-P3IA-0003").}
\thanks{Andrea Mattioni is with the Universit\'e Grenoble Alpes, CNRS, Grenoble-INP, GIPSA-lab, F-38000, Grenoble, France.
{\tt\small andrea.mattioni@grenoble-inp.gipsa-lab.fr}}
\thanks{Yongxin Wu and Yann Le Gorrec are with the AS2M department of the FEMTO-ST research lab University of Bourgogne Franche-Compt\'e, 24 rue Savary, 25000 Besan\c{c}on, France. 
{\tt\small yongxin.wu@femto-st.fr}
{\tt\small yann.le.gorrec@ens2m.fr}}
}

\maketitle

\begin{abstract}
In this technical note, we consider the stability properties of a viscously damped Timoshenko beam equation with spatially varying parameters. With the help of the port-Hamiltonian framework, we first prove the existence of solutions and show, by an appropriate Lyapunov function, that the system is exponentially stable and has an explicit decay rate. The explicit exponential bound is computed for an illustrative example of which we provide some numerical simulations.\end{abstract}


\section{Introduction}\label{sec:introduction}
\graphicspath{{./Figures/}}

The Timoshenko beam theory is often used in engineering applications to represent the propagation of vibrations in mechanical systems such as buildings, aircraft structures, flexible robots and micro grippers \cite{Mattioni2020, Ramirez2014}. In this technical note, we consider the Timoshenko beam Partial Differential Equations (PDEs) with space-varying parameters and viscous damping. In the case of constant parameters, the system has already been proven to be exponentially stable in \cite{Raposo2004}, using the Gearhart-Herbst-Prüss-Huang spectral method \cite{Huang1985}. In \cite{Raposo2004} the authors prove that there exists $M>0$ and $w>0$ such that $||T(t)z_0||\leq Me^{-wt}$ for all $z_0\in Z$, but do not suggest any estimation of these two quantities. The same result with space varying parameters has been proved in \cite{Shi2001} using the same techniques. Then, in \cite{Kim1987} the authors constructed a Lyapunov function to prove the exponential stability in case of constant parameters, but without making explicit the state's norm decay rate. Moreover, different studies focused on the stabilisation problem in the case of the presence of damping in only one beam dynamics, \textit{e.g.} vertical or rotational dynamics. In particular, in \cite{Almeida2013} the authors used a Lyapunov function to show that the system is exponentially stable if and only if the wave propagation speeds of the two dynamics are identical. A technical extension to linear and nonlinear operator equations using Lyapunov techniques can be found in \cite{Haraux1988,Zuazua1990}. Over the last twenty years, the port-Hamiltonian (PH) framework has proved to be a useful tool for stability analysis and control design for PDEs. It has been used to design static \cite{Villegas2005}, linear dynamic \cite{Ramirez2014} and nonlinear dynamic \cite{Augner2019} PDEs boundary controllers able to exponentially stabilize the origin of the closed-loop system. Existing results using the PH framework have been obtained without considering internal dissipation (\textit{e.g.} viscous damping for flexible beams). The absence of internal dissipation renders it difficult to explicitly find the exponential bound parameters, and only ``existence" results have been assessed \cite{AugnerThesis}.

Inspired by the work in \cite{Zuazua1990} and \cite{Almeida2013}, in this technical note we construct a Lyapunov function with crossing terms in order to prove the exponential stability in the case of spatially varying parameters with viscous damping in both the vertical and rotational dynamics. Moreover, the proposed Lyapunov function allows to compute the parameters $M,\; w$ of the exponential bound $||T(t)z_0||\leq Me^{-wt}$. This work relies on the PH framework \cite{Villegas2007,Jacob2012} for the result on the existence and uniqueness of solutions, and on \cite{Macchelli2004} for the state variable selection. 

The paper is organized as follows. In Section \ref{preliminaries}, we recall some technical preliminaries that will be useful for the stability proof. In Section \ref{mainresult}, is stated the main result of the paper {\it i.e.} exponential stability with an explicit formulation of the decay rate of the solution. Then, a numerical example is presented to validate the theoretical results. This technical note ends with some conclusions in Section \ref{conclusions}.

\section{Preliminaries}\label{preliminaries}
\subsection{Usefull inequalities}

Throughout the paper, we make use of some standard inequalities that are often used in the literature on the control of PDEs. We recall three classical inequalities, that hold for all functions $f,g:\Omega\rightarrow \mathbb{R}$ with $\Omega\in \mathbb{R}^N$, $N\in\mathbb{N}_{\geq 1}$:\\
\textit{Young's inequality}
\begin{equation}
fg\leq \frac{1}{2\alpha}|f|^2 + \frac{\alpha}{2}|g|^2,
\end{equation}
for all $\alpha>0$.\\
\textit{Cauchy-Schwarz inequality}
\begin{equation}
\int_0^Lf(\xi)g(\xi)d\xi \leq \left( \int_0^Lf(\xi)^2d\xi \right)^{\frac{1}{2}} \left( \int_0^Lg(\xi)^2d\xi \right)^{\frac{1}{2}}.
\end{equation}
\textit{Triangle-type inequality}
\begin{equation}
(f\pm g)^2\leq 2(|f|^2 + |g|^2).
\end{equation}
In the next lemma we introduce a Poincar\'e-type inequality that can be derived from \cite[Theorem 256]{Hardy1959}, changing the integration interval from $[0,1]$ to $[0,L]$.
\begin{lemma}[Variation of the Wirtinger's inequality]\label{lem:Wirtinger}
For any absolutely continuous function $f$ such that $f(0)=0$,
\begin{equation}
\int_0^L f(x)^2dx \leq \left(\frac{2L}{\pi}\right)^2\int_0^L\left(\frac{d}{dx}f(x)\right)^2dx.
\end{equation}
\end{lemma}

\subsection{Lyapunov stability theory}

Let $z$ belong to a Hilbert space $Z$ and consider the linear differential equation
\begin{equation}\label{eq:LinearDyn}
\dot{z} = Az\quad z(0)=z_0
\end{equation}
where we assume that the operator $A$ with domain $D(A)$ is the infinitesimal generator of a $C_0$-semigroup $T(t)$ on the state space $Z$.  In the following we denote the solution of \eqref{eq:LinearDyn} with initial condition $z_0$ as $z(t,z_0)=T(t)z_0$. Now, we introduce the concept of Lyapunov function for \eqref{eq:LinearDyn}.
\begin{definition}\label{def:LyapunovFunctional}
A continuous functional $V:Z\mapsto [0,\infty)$ is a \textit{Lyapunov functional} for \eqref{eq:LinearDyn} on $Z$ if $V(z(t,z_0))$ is Dini differentiable at $t=0$ for all $z_0\in X$ and there holds 
\begin{equation}\label{eq:DiniLyap}
\dot{V}_+(z_0):= \limsup_{t\rightarrow 0} \frac{V(z(t,z_0))-V(z_0)}{t}\leq 0.
\end{equation}\hfill $\blacksquare$
\end{definition}
Since in most practical cases, the limit \eqref{eq:DiniLyap} it is not easy to compute, we rely on Lemma 11.2.5 of \cite{Curtain2020} to establish the relation between the Dini time derivative (see Definition A.5.43 in \cite{Curtain2020}) and the Fr\'echet derivative (see Definition A.5.31 in \cite{Curtain2020}). In fact, if $V$ is Fr\'echet differentiable, then for $z\in \mathbf{D}(A)$, $V(z(t,z_0))$ is Dini differentiable and 
\begin{equation}\label{eq:WellStab:TimeDer}
\dot{V}_+(z_0):=dV(z_0)Az_0
\end{equation}
where $dV$ is the Fr\'echet derivative of $V$. In the following, we cite a part of Theorem 11.2.7 from \cite{Curtain2020}, that will be instrumental to prove exponential stability. 

\begin{theorem}\label{thm_GenExp}
Suppose that $V$ is a Lyapunov functional for \eqref{eq:LinearDyn} with $V(0)=0$. If there exist two positive constants $\kappa_1, \kappa_2>0$ such that $V(z)\geq \kappa_1 ||z||^2$ and $\dot{V}_+(z)\leq -\kappa_2 V(z)$ for all $z\in Z$, then the origin is globally exponentially stable, \textit{i.e.}
\begin{equation}
||z(t,z_0)||\leq \sqrt{\frac{V(z_0)}{\kappa_1}}e^{-\frac{\kappa_2}{2}t}.
\end{equation}
\end{theorem}

\section{Main result \label{mainresult}}

\subsection{Port Hamiltonian formulation of the Timoshenko's beam with viscous damping}\label{sec:ViscousDamping}

We consider the dynamics equations of a clamped Timoshenko beam with viscous damping
\begin{equation}\label{eq:OriginalEqs}
\begin{array}{l}
\rho \frac{\partial^2w}{\partial t^2}= \frac{\partial}{\partial \xi}\left(K\left(\frac{\partial w}{\partial \xi}- \phi\right)\right) - \gamma\frac{\partial w}{\partial t} \\
I_\rho  \frac{\partial^2\phi}{\partial t^2} = \frac{\partial}{\partial \xi}\left(EI\frac{\partial \phi}{\partial \xi}\right) + K\left(\frac{\partial w}{\partial \xi}- \phi\right) - \delta\frac{\partial \phi}{\partial t}\\
w(0,t) = \phi(0,t) = 0 \\
K(L)\left( \frac{\partial w}{\partial \xi}(L,t) - \phi(L,t)\right) = \gamma(L)\frac{dw}{dt}(L,t) \\
EI(L,t)\frac{\partial \phi}{\partial \xi} (L,t) = \delta(L)\frac{d\phi}{dt}(L,t).
\end{array}
\end{equation}
The term $\xi\in [0,L]$ identifies the spatial coordinate, while $w(\xi,t)$ and $\phi(\xi,t)$ represent the deflection and the relative rotation of a beam cross-section in the rotating frame at position $\xi$ and time $t$, respectively. $E(\xi), I(\xi)$ are the spatially dependent Young's modulus and moment of inertia of the beam's cross-section, respectively. $\rho(\xi), I_{\rho}(\xi)$ are the spatially dependent density and mass moment of inertia of the beam's cross-section, respectively. The mass moment of inertia of the cross-section is defined as $I_{\rho}(\xi)= I(\xi)\rho(\xi)$. $K(\xi)$ is defined as $K(\xi) = kG(\xi)A(\xi)$, where $k$ is a constant dependent on the shape of the cross-section  $G(\xi)$ is the shear modulus and $A(\xi)$ is the cross-sectional area. $\gamma(\xi)$ and $\delta(\xi)$ represent the space depending translating and the rotating components of the viscous damping, respectively. 
Throughout this paper, all physical parameters and their reciprocals are assumed to be absolutely continuous, positive definite and belonging to the $L_\infty([0,L])$ equivalent class of functions. Following \cite{Macchelli2004} we define the energy variables, 
\begin{equation}\label{eq:Variable definition}
z_1 = \rho \frac{\partial w}{\partial t} \quad z_2 = I_\rho \frac{\partial \phi}{\partial t} \quad z_3 = \frac{\partial w}{\partial \xi}- \phi \quad z_4 = \frac{\partial \phi}{\partial \xi}
\end{equation}
such to write the PH representation of the system with $z=[z_1\; z_2\; z_3\; z_4]^T$
\begin{equation}\label{eq:system}
\dot{z} = P_1\frac{\partial}{\partial \xi}(\mathcal{H}z) + (P_0 - G_0)(\mathcal{H}z)
\end{equation}
where,
\begin{equation}
\begin{array}{c}
P_1 = \begin{bmatrix} 0 & 0 & 1 & 0  \\ 0 & 0 & 0 & 1 \\ 1 & 0 & 0 & 0 \\ 0 & 1 & 0 & 0 \end{bmatrix} \quad P_0 = \begin{bmatrix} 0 & 0 & 0 & 0  \\ 0 & 0 & 1 & 0 \\ 0 & -1 & 0 & 0 \\ 0 & 0 & 0 & 0 \end{bmatrix}\\
\mathcal{H} = \begin{bmatrix}\frac{1}{\rho} & 0 & 0 & 0 \\ 0 & \frac{1}{I_\rho} & 0 & 0 \\ 0 & 0 & K & 0 \\ 0 & 0 & 0 & EI \end{bmatrix}\quad 
G_0 = \begin{bmatrix}\gamma & 0 & 0 & 0 \\ 0 & \delta & 0 & 0 \\ 0 & 0 & 0 & 0 \\ 0 & 0 & 0 & 0 \end{bmatrix}.
\end{array}
\end{equation}
We define the state space $Z=L_2([0,L], \mathbb{R}^4)$ and we equip it with the energy inner product
\begin{equation}
\langle z_1, z_2\rangle_Z = \langle z_1, \mathcal{H}z_2\rangle_{L_2} = \int_0^Lz_1^T\mathcal{H}z_2d\xi
\end{equation}
such that the state's norm is defined as $||z||_Z=\sqrt{\langle z, z\rangle_Z}$ while the energy writes
\begin{equation}
    E=\frac{1}{2}\langle z, z\rangle_Z.
\end{equation}
Following \cite{LeGorrec2005}, we define the boundary flow and effort as a composition of the co-energy variables at the boundary of the spatial domain
\begin{equation}
\begin{bmatrix} f_{\partial}(t) \\ e_{\partial}(t) \end{bmatrix} = \frac{1}{\sqrt{2}}\begin{bmatrix} P_1 & -P_1 \\ I & I \end{bmatrix}\begin{bmatrix} (\mathcal{H}z)(0,t) \\ (\mathcal{H}z)(L,t) \end{bmatrix}.
\end{equation}
The boundary flow and effort are instrumental to define the boundary operators such to obtain a well-posed (in the Hadamard sense) set of PDEs
\begin{equation}
\begin{array}{c}
\mathcal{B}_1z(t) = W_{B1}\begin{bmatrix} f_{\partial}(t) \\ e_{\partial}(t) \end{bmatrix} =\begin{bmatrix} \frac{1}{\rho(0)}z_1(0,t) \\ \frac{1}{I_\rho(0)}z_2(0,t) \end{bmatrix}  \\ 
\mathcal{B}_2z(t) = W_{B2}\begin{bmatrix} f_{\partial}(t) \\ e_{\partial}(t) \end{bmatrix} = \begin{bmatrix} -K(L)z_3(L,t) \\ -EI(L)z_4(L,t)
\end{bmatrix} \\ 
\mathcal{C}_1z(t) = W_{C1}\begin{bmatrix} f_{\partial}(t) \\ e_{\partial}(t) \end{bmatrix} =\begin{bmatrix} K(0)z_3(0,t) \\ EI(0)z_4(0,t)
\end{bmatrix} \\
\mathcal{C}_2z(t) = W_{C2}\begin{bmatrix} f_{\partial}(t) \\ e_{\partial}(t) \end{bmatrix} = \begin{bmatrix}\frac{1}{\rho(L)}z_1(L,t) \\ \frac{1}{I_\rho(L)}z_2(L,t) \end{bmatrix} 
\end{array}
\end{equation}
with
\begin{equation}
\begin{array}{ll}
W_{B1} = -\frac{1}{\sqrt{2}}\begin{bsmallmatrix} 0 & 0 & 1 & 0 & 1 & 0 & 0 & 0 \\ 0 & 0 & 0 & 1 & 0 & 1 & 0 & 0 \end{bsmallmatrix} & W_{B2} = \frac{1}{\sqrt{2}}\begin{bsmallmatrix} 1 & 0 & 0 & 0 & 0 & 0 & -1 & 0 \\ 0 & 1 & 0 & 0 & 0 & 0 & 0 & -1 \end{bsmallmatrix} \\
W_{C1} = \frac{1}{\sqrt{2}}\begin{bsmallmatrix} 1 & 0 & 0 & 0 & 0 & 0 & 1 & 0 \\ 0 & 1 & 0 & 0 & 0 & 0 & 0 & 1  \end{bsmallmatrix} & W_{C2} = \frac{1}{\sqrt{2}}\begin{bsmallmatrix} 0 & 0 & -1 & 0 & 1 & 0 & 0 & 0 \\ 0 & 0 & 0 & -1 & 0 & 1 & 0 & 0 \end{bsmallmatrix}.
\end{array}
\quad 
\end{equation}
We can now define the operator
\begin{equation}\label{eq:pHoperator}
\mathcal{J}z = P_1\frac{\partial}{\partial \xi}(\mathcal{H}z) + (P_0 - G_0)(\mathcal{H}z)
\end{equation}
with domain
\begin{equation}\label{eq:pHoperatorDomain}
D(\mathcal{J}) = \{z\in Z\mid \mathcal{H}z\in H^1, \; \mathcal{B}_1z= 0,  \mathcal{B}_2z = -S(L)\mathcal{C}_2z\}
\end{equation}
and $S = \text{diag}\{\gamma, \delta\}$. In the following proposition, we show that the operator $\mathcal{J}$ with domain $D(\mathcal{J})$ generates a contraction $C_0$-semigroup,  or equivalently that the dynamical system \eqref{eq:system} is well-posed.

\begin{proposition}
The operator $\mathcal{J}$ in \eqref{eq:pHoperator} with domain \eqref{eq:pHoperatorDomain} generates a contraction $C_0$-semigroup on the state space $Z$. Moreover, it is true that
\begin{equation}\label{eq:EnergyTimeDer}
\begin{array}{rl}
\dot{E}_+ = & \langle \mathcal{J}z, z \rangle_Z \\
= & -\mathlarger{\int}_0^L\left\{\frac{\gamma}{\rho^2}z_1^2  +\frac{\delta}{I_\rho^2}z_2^2\right\} d\xi - (\mathcal{C}_2z)^TS(L)\mathcal{C}_2z.
\end{array}
\end{equation}
\end{proposition}
\begin{proof}
For the generation result, is sufficient to use Theorem 6.9 of \cite{Villegas2007}. For the energy time derivative, we compute
\begin{equation}
\begin{array}{rl}
\dot{E}_+(z) = & dE(z)\mathcal{J}z = \langle \mathcal{J}z, z \rangle_Z \\
= & \mathlarger{\int}_0^L\left( P_1\frac{\partial}{\partial \xi}(\mathcal{H}z) + (P_0 - G_0)(\mathcal{H}z)\right)^T\mathcal{H}zd\xi \\
= & -\mathlarger{\int}_0^L(\mathcal{H}z)^TG_0(\mathcal{H}z)d\xi \\ & + \mathlarger{\int}_0^L\left( P_1\frac{\partial}{\partial \xi}(\mathcal{H}z) + P_0(\mathcal{H}z)\right)^T(\mathcal{H}z)d\xi .
\end{array}
\end{equation}
We notice that the first term of the last equation corresponds to the first term in \eqref{eq:EnergyTimeDer}, while the second term, after integration by parts, makes appear the second term in \eqref{eq:EnergyTimeDer}.
\end{proof}

Next, we present two inequalities that will be useful in the consequent stability analysis with Lyapunov arguments.

\begin{lemma}\label{lem:TIMOinequalities}
For any function $z_3,z_4\in L_2([0,L], \mathbb{R})$, the following inequalities hold through
\begin{equation}
\int_0^L\left(\int_0^\xi Kz_3ds\right)^2d\xi \leq  k_1\int_0^LKz_3^2d\xi,
\end{equation}
\begin{equation}
\int_0^L\left(\int_0^\xi EIz_4ds\right)^2d\xi \leq k_2\int_0^LEIz_4^2d\xi.
\end{equation}
with $k_1 = \left(\frac{2L}{\pi}\right)^2\bar{K}$ and $k_2 = \left(\frac{2L}{\pi}\right)^2\bar{EI}$, where $\bar{K} = \esssup \limits_{\xi\in[0,L]} K(\xi)$ and $\bar{EI} = \esssup \limits_{\xi\in[0,L]} EI(\xi)$.
\end{lemma}
\begin{proof}
To obtain the first inequality we apply Wirtinger's inequality of Lemma \ref{lem:Wirtinger}
\begin{equation}
\begin{array}{rl}
\int_0^L\left(\int_0^\xi Kz_3ds\right)^2d\xi \leq & \left(\frac{2L}{\pi}\right)^2 \int_0^L (Kz_3)^2 d\xi \\
\leq & \left(\frac{2L}{\pi}\right)^2\bar{K}\int_0^L K(z_3)^2 d\xi.
\end{array}
\end{equation}
The second inequality can be obtained in exactly the same manner.
\end{proof}

\subsection{Stability analysis}
The aim of this section is to find an appropriate Lyapunov function allowing to show the exponential stability of the system and to explicit its decay rate. The Lyapunov function will be composed of the natural energy of the system together with two cross-coupling terms. More precisely we define the Lyapunov function as
\begin{equation}\label{eq:LyapunovFunc}
V = n_0E + n_1F_1 + n_2F_2
\end{equation}
with $n_1,n_2>0$ while $F_1,F_2$ are defined as
\begin{equation}\label{eq:Cross_functions}
\begin{array}{ll}
F_1 = \mathlarger{\int}_0^Lz_1\left(\mathlarger{\int}_0^\xi Kz_3ds\right)d\xi, &
F_2 = \mathlarger{\int}_0^Lz_2\left(\mathlarger{\int}_0^\xi EIz_4ds\right)d\xi
\end{array}
\end{equation}
\begin{lemma}
For any state $z\in Z$ the Lyapunov function \eqref{eq:LyapunovFunc} is well-defined, \textit{i.e.} it is finite in all the state space $Z$.
\end{lemma}
\begin{proof}
The energy term $E$ in \eqref{eq:LyapunovFunc} is bounded as soon as $z\in Z$. The function $F_1$ can be bounded by using firstly the \textit{Young's inequality} and secondly Lemma \ref{lem:TIMOinequalities}
\begin{equation}
\begin{array}{rl}
\mathlarger{\int}_0^Lz_1\left(\mathlarger{\int}_0^\xi Kz_3ds\right)d\xi \leq &\displaystyle \frac{1}{2}\int_0^L \left(\int_0^\xi Kz_3 ds\right)^2d\xi  \\ & \displaystyle +\frac{1}{2}\int_0^Lz_1^2d\xi \\
\displaystyle \leq & \displaystyle \frac{1}{2}k_1\int_0^L Kz_3^2d\xi + \frac{1}{2}\int_0^Lz_1^2d\xi
\end{array}
\end{equation}
which is bounded as soon as $z\in Z$. The term $F_2$ can be bounded in a very similar manner.
\end{proof}
Since the objective of this Lyapunov study is to obtain an inequality of the type $\dot{V}_+\leq -\kappa_2 V$,  the choice of the $F_1, F_2$ crossing terms is justified by the need of making appear the missing negative square terms in the time derivative of the Lyapunov functional. Similarly, as in \cite{Almeida2013}, the general idea comes from the fact that for $i\in \{ 1,2,3,4\}$
\begin{equation}
\displaystyle\int_0^L\frac{\partial z_i}{\partial \xi}\left(\int_0^\xi z_i ds\right)d\xi = \displaystyle\left[z_i\int_0^\xi z_ids\right]_0^L  \displaystyle - \int_0^Lz_i^2d\xi.
\end{equation}

In the next proposition, we show that the functional $V$ is positive definite and bounded by the energy if the constants $n_0,n_1,n_2$ are chosen appropriately.

\begin{proposition}\label{prop:Lyap_Ineqs}
For all $n_0,n_1,n_2>0$, the Lyapunov function $V$ in \eqref{eq:LyapunovFunc} is such that:
\begin{enumerate}[i)]
\item $V(z)\geq \kappa_1 ||z||^2$ for all $z\in Z$, with $\kappa_1 = \min\{\left(\frac{n_0}{2} - \frac{n_1\bar{\rho}}{2} \right), \left(\frac{n_0}{2} - \frac{n_2\bar{I}_\rho}{2} \right), \left(\frac{n_0}{2}  - \frac{n_1k_1}{2} \right), \left(\frac{n_0}{2} - \frac{n_2k_2}{2}\right)\}$, with $\bar{\rho} = \esssup \limits_{\xi\in [0,L]} \rho(\xi)$ and $\bar{I}_\rho = \esssup \limits_{\xi\in [0,L]} I_\rho(\xi)$.
\item $V(z)\leq \eta E$ for all $z\in Z$, with $\eta=\max\{\left(n_0 + n_1\bar{\rho} \right),\left(n_0 + n_2\bar{I}_\rho \right),\left(n_0 + n_1k_1 \right),\left(n_0  + n_2k_2 \right)\}$.
\end{enumerate}
\end{proposition}
\begin{proof}
i) We apply \textit{Young's inequality} (with $\alpha=1$ and $f$ replaced with $-f$) to get
\begin{equation}\label{eq:Lyap_lowbound}
\begin{array}{rl}
V\geq & \mathlarger{\int}_0^L\left\{ \left(\frac{n_0}{2} - \frac{n_1\rho}{2}\right)\frac{z_1^2}{\rho} + \left(\frac{n_0}{2} - \frac{n_2I_\rho}{2}\right)\frac{z_2^2}{I_\rho}\right. \\
& \frac{n_0}{2} Kz_3^2 + \frac{n_0}{2} EIz_4^2 - \frac{n_1}{2}\left(\mathlarger{\int}_0^\xi Kz_3ds\right)^2 \\ & \left. - \frac{n_2}{2}\left(\mathlarger{\int}_0^\xi EIz_4ds\right)^2\right\} d\xi \\
\geq & \mathlarger{\int}_0^L \overbrace{\mathsmaller{\left(\frac{n_0}{2} - \frac{n_1\bar{\rho}}{2} \right)}}^{a_1}\frac{z_1^2}{\rho} + \overbrace{\mathsmaller{\left(\frac{n_0}{2} - \frac{n_2\bar{I}_\rho}{2} \right)}}^{a_2}\frac{z_2^2}{I_\rho} \\
&  \overbrace{\mathsmaller{\left(\frac{n_0}{2}  - \frac{n_1k_1}{2} \right)}}^{a_3}Kz_3^2 + \overbrace{\mathsmaller{\left(\frac{n_0}{2} - \frac{n_2k_2}{2}\right)}}^{a_4}EIz_4^2 d\xi
\end{array}
\end{equation}
where Lemma \ref{lem:TIMOinequalities} has been applied to obtain the second inequality. 
Defining $\kappa_1 = \min\{a_1,a_2,a_3,a_4\}$ we obtain the inequality of item i).\\
ii) We apply \textit{Cauchy-Swartz} and \textit{Young's Inequalities} with $\alpha=1$ to get
\begin{equation}\label{eq:Lyap_upbound}
\begin{array}{rl}
V\leq & \mathlarger{\int}_0^L \left\{ \left(\frac{n_0}{2} + \frac{n_1\rho}{2} \right)\frac{z_1^2}{\rho} + \left(\frac{n_0}{2} + \frac{n_2I_\rho}{2} \right)\frac{z_2^2}{I_\rho} \right.\\
& \frac{n_0}{2} Kz_3^2 + \frac{n_0}{2} EIz_4^2 + \frac{n_1}{2}\left(\mathlarger{\int}_0^\xi Kz_3ds\right)^2 \\ & \left. + \frac{n_2}{2}\left(\mathlarger{\int}_0^\xi EIz_4ds\right)^2\right\} d\xi \\
\leq & \frac{1}{2}\mathlarger{\int}_0^L \overbrace{\mathsmaller{\left(n_0 + n_1\bar{\rho} \right)}}^{b_1}\frac{z_1^2}{\rho} + \overbrace{\mathsmaller{\left(n_0 + n_2\bar{I}_\rho \right)}}^{b_2}\frac{z_2^2}{I_\rho}\\
& \overbrace{\mathsmaller{\left(n_0 + n_1k_1 \right)}}^{b_3}Kz_3^2 + \overbrace{\mathsmaller{\left(n_0  + n_2k_2 \right)}}^{b_4}EIz_4^2d\xi,
\end{array}
\end{equation}
where Lemma \ref{lem:TIMOinequalities} has been applied to obtain the second inequality. We define the constant $\eta=\max\{b_1,b_2,b_3,b_4\}$ to obtain the inequality of item ii).
\end{proof}

In the following theorem, we present the main result of this paper, \textit{i.e.} we show the exponential stability of the Timoshenko beam model with viscous damping making use of the Lyapunov function in \eqref{eq:LyapunovFunc}.

\begin{theorem}\label{thm:main}
Consider the Timoshenko's beam equation with space-varying parameters \eqref{eq:system} and the Lyapunov functional $V$ in \eqref{eq:LyapunovFunc}, then the norm of the $C_0$-semigroup generated by the operator \eqref{eq:pHoperator}-\eqref{eq:pHoperatorDomain} can be bounded by
\begin{equation}\label{eq:estimation}
||z(t,z_0)||\leq \sqrt{\frac{V(z_0)}{\kappa_1}}e^{-\frac{\kappa_2}{2}t}
\end{equation}
where $\kappa_1>0$ is defined in point i) of Proposition \ref{prop:Lyap_Ineqs}, $\kappa_2 = \frac{\beta}{\eta} >0$ with $\eta$ defined in point ii) of Proposition \ref{prop:Lyap_Ineqs} and $\beta = \min\{\underline{c}_1,\underline{c}_2,\underline{c}_3,\underline{c}_4\}>0$ with $c_i$ defined in \eqref{eq:functions_ci} and $\underline{c}_i = \essinf\limits_{\xi\in [0,L]} c_i(\xi)$ with $i\in \{1,2,3,4\}$.
\end{theorem}
\begin{proof}
We start by computing the estimates of the Dini's time derivative of the functionals $F_1,\; F_2$ composing the Lyapunov functional in \eqref{eq:LyapunovFunc}
\begin{equation}
\begin{array}{rl}
\dot{F}_{1,+} = &  \mathlarger{\int}_0^L\left\{\left( \frac{\partial }{\partial \xi}\left( Kz_3\right) - \frac{\gamma}{\rho}z_1 \right)\left( \mathlarger{\int}_0^\xi Kz_3ds \right) \right.  \\ & \hfill \left.+ z_1\left(\mathlarger{\int}_0^\xi K\left(\frac{\partial}{\partial s}\left(\frac{z_1}{\rho}\right) - \frac{z_2}{I_\rho}\right)ds\right) \right\}d\xi\\
= &   \mathlarger{\int}_0^L\left\{ \frac{\partial}{\partial \xi}(Kz_3)\left(\mathlarger{\int}_0^\xi Kz_3 ds\right) - \frac{\gamma}{\rho}z_1\left(\mathlarger{\int}_0^\xi Kz_3ds\right) \right. \\ &  \left.+z_1\mathlarger{\int}_0^\xi K\frac{\partial}{\partial s}\left(\frac{z_1}{\rho}\right)ds - z_1\left(\mathlarger{\int}_0^\xi\frac{K}{I_\rho}z_2ds\right) \right\}d\xi.
\end{array} 
\end{equation}
We apply integration by parts on the first and third terms while using \textit{Cauchy-Schwartz} in the second and fourth terms
\begin{equation}
\begin{array}{rl}
\dot{F}_{1,+} \leq &    \left[Kz_3\mathlarger{\int}_0^\xi Kz_3ds \right]_0^L -\mathlarger{\int}_0^L\left(Kz_3 \right)^2d\xi \\
&  + \left(\mathlarger{\int}_0^L\left(\frac{\gamma}{\rho}z_1\right)^2d\xi\right)^{\frac{1}{2}}\left(\mathlarger{\int}_0^L\left(\mathlarger{\int}_0^\xi Kz_3ds\right)^2d\xi \right)^{\frac{1}{2}} \\ &   +\mathlarger{\int}_0^Lz_1\left(\left[\frac{K}{\rho}z_1\right]_0^\xi - \mathlarger{\int}_0^\xi\frac{z_1}{\rho}\frac{d K}{d s} ds\right) d\xi \\
&   + \left(\mathlarger{\int}_0^L z_1^2d\xi\right)^{\frac{1}{2}}\left(\mathlarger{\int}_0^L\left(\mathlarger{\int}_0^\xi \frac{K}{I_\rho}z_2ds\right)^2d\xi \right)^{\frac{1}{2}}.
\end{array}
\end{equation}
We define the parameter $K_d = \frac{d K}{d s}$ while using Lemma \ref{lem:Wirtinger} and the \textit{Young's inequality} to obtain
\begin{equation}
\begin{array}{rl}
\dot{F}_{1,+} \leq &   K(L)z_3(L,t)\mathlarger{\int}_0^LKz_3d\xi - \mathlarger{\int}_0^L(Kz_3)^{2}d\xi \\ & + \left( \mathlarger{\int}_0^L\left(\frac{\gamma}{\rho}z_1\right)^2d\xi\right)^\frac{1}{2} \left(\left(\frac{2L}{\pi}\right)^2\mathlarger{\int}_0^L\left(Kz_3 \right)^2d\xi\right)^{\frac{1}{2}} \\ 
&    +\mathlarger{\int}_0^L\frac{K}{\rho}z_1^2d\xi - \mathlarger{\int}_0^Lz_1\frac{K(0)}{\rho(0)}z_1(0,t)d\xi  \\
&   - \mathlarger{\int}_0^Lz_1 \left( \mathlarger{\int}_0^\xi \frac{K_d}{\rho}z_1ds\right) d\xi + \left(\mathlarger{\int}_0^L z_1^2d\xi\right)^{\frac{1}{2}} \\ &  \cdot \left(\left(\frac{2L}{\pi}\right)^2\mathlarger{\int}_0^L\left(\frac{K}{I_\rho}z_2 \right)^2d\xi\right)^{\frac{1}{2}}
\end{array}
\end{equation}
then, using again the \textit{Young's inequality} together with \textit{Cauchy-Schwartz}, Lemma \ref{lem:Wirtinger} and the boundary conditions $\mathcal{B}_1z = 0$ we get
\begin{equation}\label{eq: F1 derivative}
\begin{array}{rl}
\dot{F}_{1,+} \leq   &  \frac{L}{2}(K(L)z_3(L,t))^2 + \mathlarger{\int}_0^L \left\{\frac{1}{2}(Kz_3)^2-(Kz_3)^2 \right. \\ &  + \frac{\alpha_1}{2}\left(\frac{\gamma}{\rho}z_1\right)^2  + \frac{1}{2\alpha_1}\left(\frac{2L}{\pi} \right)^2\left(Kz_3\right)^2  \\ 
& + \left. \frac{K}{\rho}z_1^2 + \frac{1}{2}z_1^2 + \frac{1}{2}\left(\frac{2L}{\pi}\right)^2\left(\frac{K}{I_\rho}z_2\right)^2\right\} d\xi \\
&   +\left(\mathlarger{\int}_0^Lz_1^2d\xi\right)^{\frac{1}{2}}\left(\mathlarger{\int}_0^L\left(\mathlarger{\int}_0^\xi \frac{K_d}{\rho}z_1ds\right)^2d\xi\right)^{\frac{1}{2}}\\
\leq &   \mathlarger{\int}_0^L \left\{ \left( \frac{\alpha_1\gamma^2}{2\rho} + K + \rho +\frac{1}{2\rho} \left(\frac{2LK_d}{\pi}\right)^2\right)\frac{z_1^2}{\rho} \right. \\
&+ \left. \frac{1}{2I_\rho}\left( \frac{2LK}{\pi}\right)^2\frac{z_2^2}{I_\rho} - \left(\frac{K}{2} -\frac{K}{2\alpha_1} \left(\frac{2L}{\pi} \right)^2\right)Kz_3^2 \right\} d\xi \\
&  + \frac{L}{2}(K(L)z_3(L,t))^2.
\end{array}
\end{equation}
With a very similar procedure as for $F_1$ we bound the $F_2$ time derivative with
\begin{equation}\label{eq: F2 derivative}
\begin{array}{rl}
\dot{F}_{2,+} = &   \mathlarger{\int}_0^L\left(\frac{\partial}{\partial \xi}\left(EIz_4\right) - Kz_3 - \frac{\delta}{I_\rho}z_2\right)\left(\mathlarger{\int}_0^\xi EIz_4ds\right) \\ &   \hfill + z_2\left(\mathlarger{\int}_0^\xi EI\frac{\partial}{\partial \xi}\left(\frac{1}{I_\rho}z_2\right)ds\right)d\xi \\
\leq &   \frac{L}{2}(EI(L)z_4(L,t))^2 + \mathlarger{\int}_0^L\left\{\frac{1}{2}(EIz_4)^2 - (EIz_4)^2 \right.\\ & + \frac{\alpha_2}{2}(Kz_3)^2  +\frac{(2L)^2}{2\alpha_2\pi^2}(EIz_4)^2 +\frac{\alpha_3}{2}\left(\frac{\delta}{I_\rho}z_2\right)^2  \\
&   +\frac{(2L)^2}{2\alpha_3\pi^2}(EIz_4)^2 +\frac{EI}{I_\rho}z_2^2  + \frac{1}{2}z_2^2 \\
&  \left. + \frac{1}{2}\left(\frac{2L}{\pi}\right)^2\left(\frac{EI_d}{I_\rho}z_2\right)^2\right\} d\xi  \\
 \leq &  \mathlarger{\int}_0^L \left\{\left( \frac{\alpha_3\delta^2}{2I_\rho}+ EI+ \frac{I_\rho}{2} + \frac{1}{2I_\rho}\left(\frac{2LEI_d}{\pi }\ \right)^2 \right)\frac{z_2^2}{I_\rho} \right.\\
&   +\frac{\alpha_2K}{2}Kz_3^2 -  \left( \frac{EI}{2} - \frac{EI(2L)^2}{2\alpha_2\pi^2} - \frac{EI(2L)^2}{2\alpha_3\pi^2} \right)\\
&  \left. \cdot EIz_4^2 \right\} d\xi  + \frac{L}{2}(EI(L)z_4(L,t))^2 
\end{array}
\end{equation}
where $EI_d = \frac{dEI}{d\xi}$ and $\alpha_1, \alpha_2, \alpha_3>0$ are constants to be determined later. We replace \eqref{eq: F1 derivative}, \eqref{eq: F2 derivative} and \eqref{eq:EnergyTimeDer} in the Lyapunov function's time derivative
\begin{equation}
\dot{V}_+ = n_0\dot{E}_+ + n_1\dot{F}_{1,+} + n_2\dot{F}_{2,+}
\end{equation}
and considering $\mathcal{B}_1z=0$, $\mathcal{B}_2z = -S(L)\mathcal{C}_2z$ we obtain
\begin{equation}\label{eq:LyapInequality}
\begin{array}{rl}
\dot{V}_{+} \leq & -\mathlarger{\int}_0^L\left\{ c_1\frac{z_1^2}{\rho} + c_2\frac{z_2}{I_\rho} + c_3Kz_3^2 + c_4EIz_4^2\right\} d\xi \\ 
& - c_5\left(\frac{z_1(L,t)}{\rho(L)}\right)^2  - c_6\left(\frac{z_2(L,t)}{I_\rho(L)}\right)^2 
\end{array}
\end{equation}
with functions
\begin{equation}\label{eq:functions_ci}
\begin{array}{rl}
c_1 = & \frac{n_0\gamma}{\rho^2} - \frac{n_1\alpha_1\gamma^2}{2\rho} - n_1K - n_1\rho - \frac{n_1}{2\rho} \left(\frac{2LK_I}{\pi}\right)^2 \\
c_2 = & \frac{n_0\delta}{I_\rho^2}-\frac{n_1}{2I_\rho}\left( \frac{2LK}{\pi}\right)^2 - \frac{n_2\alpha_3\delta^2}{2I_\rho} - n_2EI - n_2\frac{I_\rho}{2} \\
& - \frac{n_2}{2I_\rho}\left(\frac{2LEI_I}{\pi } \right)^2 \\
c_3 = & \frac{n_1K}{2} - \frac{n_1K}{2\alpha_1} \left(\frac{2L}{\pi} \right)^2 - \frac{n_2\alpha_2 K}{2} \\
c_4 = & \frac{n_2EI}{2} - \frac{n_2EI(2L)^2}{2\alpha_2\pi^2} - \frac{n_2EI(2L)^2}{2\alpha_3\pi^2} \\
c_5 = & n_0\gamma(L)  - \frac{L n_1\gamma(L)^2}{2}   \\
c_6 = & n_0\delta(L) - \frac{L n_2\delta(L)^2}{2} .
\end{array}
\end{equation}
Then, the constants $n_0,n_1,n_2$ and $\alpha_1,\; \alpha_2,\; \alpha_3$ could be chosen as following
\begin{enumerate}
\item Fix an arbitrary $n_2 > 0$.
\item Select $\alpha_2, \alpha_3$ sufficiently large to obtain $c_4>0 \; \forall \xi\in [0,L]$.
\item Select $\alpha_1$ and $n_1$ sufficiently large such that $c_3>0 \; \forall \xi\in [0,L]$.
\item The constant $n_0$ is selected sufficiently large such that $c_1,c_2,c_5,c_6>0 \; \forall \xi\in [0,L]$ and $\kappa_1$ of point \textit{i)} of Proposition \ref{prop:Lyap_Ineqs} is strictly positive $\kappa_1>0$.
\end{enumerate}
Therefore we have
\begin{equation}
\dot{V}_+\leq -\beta E
\end{equation}
with $\beta$ defined in the Theorem's statement. Using point \textit{ii)} of Proposition \ref{prop:Lyap_Ineqs} we obtain
\begin{equation}
\dot{V}_+\leq -\kappa_2 V
\end{equation}
with $\kappa_2 = \frac{\beta}{\eta}$. Hence, using Theorem \ref{thm_GenExp}, we can conclude that the origin is an exponentially stable equilibrium, and the trajectories of system \eqref{eq:system} fulfil the estimation \eqref{eq:estimation}.
\end{proof}

\begin{remark}
    The boundary conditions at $\xi = 0$ and $\xi = L$ can be interchanged without changing the result of Theorem~\ref{thm:main}. 
\end{remark}

\begin{remark}
In case of constant parameters $\rho, I_\rho, K, EI$ it is possible to prove that the Dini time derivative of the cross-term functions in \eqref{eq:Cross_functions} becomes
\begin{equation}
\begin{array}{rl}
\dot{F}_{1,+} \leq & \int_0^L\left\{ \left(\frac{\alpha_1\gamma^2}{2\rho} + K + \rho\right)\frac{z_1^2}{\rho}  + \frac{1}{2I_\rho}\left(\frac{2LK}{\pi}\right)^2\frac{z_2^2}{I_\rho} \right. \\ & \left. - \left( \frac{K}{2} - \frac{K}{2\alpha_1}\left(\frac{2L}{\pi}\right)^2\right)Kz_3^2\right\}d\xi \\ & + \frac{L}{2}\left(Kz_3(L,t)\right)^2
\end{array}
\end{equation}
\begin{equation}
\begin{array}{rl}
\dot{F}_{2,+} \leq & \int_0^L \left\{ \left(\frac{\alpha_3\delta^2}{2I_\rho} + \frac{3}{2}EI\right)\frac{z_2^2}{I_\rho} + \frac{\alpha_2K}{2}Kz_3^2 \right. \\ & \left. - \left(\frac{EI}{2} - \frac{EI(2L)^2}{2\alpha_2\pi^2} - \frac{EI(2L)^2}{2\alpha_3\pi^2}  \right)EIz_4^2 \right\}d\xi \\
&  + \frac{L}{2}(EIz_4(L,t))^2.
\end{array}
\end{equation}
Therefore, the Dini time derivative of the Lyapunov function takes the same form as in \eqref{eq:LyapInequality}, but with constant coefficients
\begin{equation}\label{eq:constant coefficients}
\begin{array}{rl}
c_1 = & \frac{n_0\gamma}{\rho^2} - \frac{n_1\alpha_1\gamma^2}{2\rho} - n_1K - n_1\rho  \\
c_2 = & \frac{n_0\delta}{I_\rho^2}-\frac{n_1}{2I_\rho}\left( \frac{2LK}{\pi}\right)^2 - \frac{n_2\alpha_3\delta^2}{2I_\rho} - \frac{3n_2}{2}EI  \\
c_3 = & \frac{n_1K}{2} - \frac{n_1K}{\alpha_1} \left(\frac{2L}{\pi} \right)^2 - \frac{n_2\alpha_2 K}{2} \\
c_4 = & \frac{n_2EI}{2} - \frac{n_2EI(2L)^2}{2\alpha_2\pi^2} - \frac{n_2EI(2L)^2}{2\alpha_3\pi^2} \\
c_5 = & n_0\gamma(L)  - \frac{L n_1\gamma(L)^2}{2}   \\
c_6 = & n_0\delta(L) - \frac{L n_2\delta(L)^2}{2}.
\end{array}
\end{equation}
\end{remark}
\vspace{0.3cm}

The explicit value of the exponential decay rate $\kappa_2$ depends on the coefficients $n_0,n_1,n_2$ as well as on $\alpha_1,\alpha_2,\alpha_3$. Given a certain set of values of the physical parameters $\rho, I_\rho, K, EI, L, \gamma, \delta$, different values of the exponential decrease rate can be obtained by varying $n_0,n_1,n_2,\alpha_1,\alpha_2,\alpha_3$ as soon as the positive conditions of $\kappa_1, \beta$ and $\eta$ are respected.

\begin{example}
Assume that the Timoshenko's beam equation in \eqref{eq:OriginalEqs} have a length $L=1$ and the parameters $\rho,\; I_\rho, \; K\; EI, \; \gamma, \; \delta$ have the following shape
\begin{equation}
    (\cdot) = 0.4 + 0.01\sin( 2\pi\xi + \phi_{(\cdot)}),
\end{equation}
with
\begin{equation*}
    \phi_\rho = \frac{\pi}{4} \quad \phi_{I_\rho} = \frac{3\pi}{4} \quad \phi_K = \frac{\pi}{6} \quad \phi_{EI} =  \frac{2\pi}{3} \quad  \phi_\gamma = 0 \quad  \phi_\delta = \frac{\pi}{2}.
\end{equation*}
Consider the Lyapunov function in \eqref{eq:LyapunovFunc} with constants $n_0 = 37 ,\; n_1 = 67, \; n_2 = 39 $ and $\alpha_1 = 5 ,\; \alpha_2 = 1, \; \alpha_3 = 6$. Therefore, according to Theorem~\ref{thm:main}, we can compute the exponential bound \eqref{eq:estimation} coefficients $\kappa_1 = 4.77$ and $\kappa_2 = \frac{\beta}{\eta} = \frac{4.01}{64.47} = 0.0622$. \\
In order to show the exponential bound of the system's state norm, we perform the numerical simulations using the Matlab\textsuperscript{\textregistered} environment and the ``ode23tb" time integration algorithm. To do that, a PH structure-preserving finite element spatial discretization as described in \cite[Section 2.2]{MattioniThesis} has been carried on \eqref{eq:system} to obtain a finite dimensional Linear Time Invariant (LTI) PH approximation of \eqref{eq:system}. In this specific example, the system has been divided into $50$ discretizing elements; therefore, the LTI system has $200$ states. To perform the numerical simulations, we impose the initial conditions $z_1(\xi,0) = z_2(\xi,0) = 0$ and $z_3 = \frac{1}{2}(1-\cos(\frac{2\pi\xi}{L}))$, $z_4 = 1-\cos(\frac{2\pi\xi}{L})$. Figure \ref{fig:1} shows the trajectory time evolution of the beam deformation $w(\xi,t)$ and its velocity $\dot{w}(\xi,t)$, while Figure \ref{fig:2} shows the state's norm evolution together with the computed exponential bound \eqref{eq:estimation}. We remark that the computed exponential bound is conservative. This is because the proposed Lyapunov parameters are not optimal with respect to the maximum decay rate.

\begin{figure}
    \centering
    \includegraphics{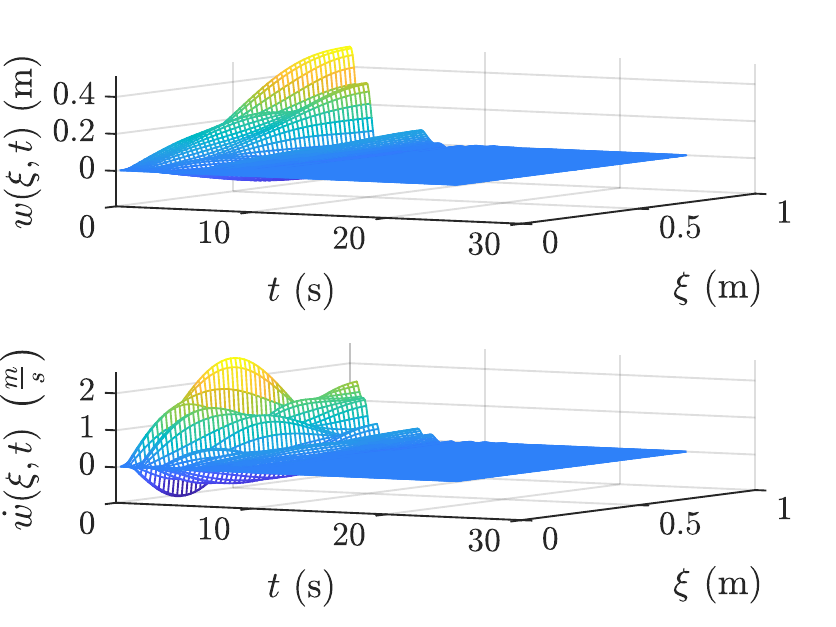}
    \caption{$w(\xi,t)$ and $\dot{w}(\xi,t)$ evolution along time.}\label{fig:1}
\end{figure}{}

\begin{figure}
    \centering
    \includegraphics{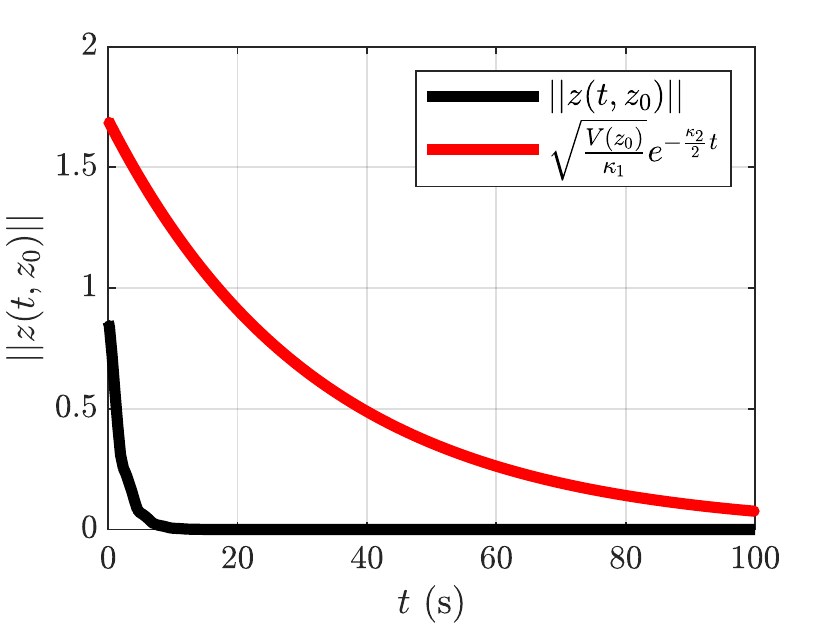}
    \caption{State's norm evolution along time and exponential bound with current parameters selection.}\label{fig:2}
\end{figure}{}

\end{example}

\section{Conclusions}\label{conclusions}

In this paper, the exponential stability problem of Timoshenko's beam equations with space-varying parameters and with viscous damping in both the vertical and rotational dynamics has been considered. After recalling some basic inequalities, Timoshenko's equations have been rewritten in the PH framework and the existence and uniqueness of solutions have been proven. The exponential bound of the state norm has been obtained using Lyapunov arguments. The defined Lyapunov function is composed of the internal energy and two crossing terms and it has been proven to be finite in all the state space. Therefore, the time derivative of the Lyapunov function along the system trajectories has been computed, and the exponential stability has been proven. For sake of generality, the Lyapunov function's parameters have not been a priori fixed. In an illustrative example, the exponential bound coefficients are computed for Timoshenko's beam equations with space-varying parameters.\\
The future work will focus on the stabilization problem in case the viscously damped flexible beam is part of a mechanism. For this purpose, the Lyapunov function proposed in this technical note can be used, in composition with other terms, to prove exponential stability. 

\bibliographystyle{unsrt}

\begin{thebibliography}{}

\end{thebibliography}


\begin{thebibliography}{99}

\bibitem{Mattioni2020} A. Mattioni, Y. Wu and Y. Le Gorrec, ``Infinite dimensional model of a double flexible-link manipulator: The Port-Hamiltonian approach," \textit{Applied Mathematical Modelling}, vol. 83, pp. 59-75, 2020.

\bibitem{Ramirez2014} H. Ramírez, Y. Le Gorrec, A. Macchelli and H. Zwart, ``Exponential Stabilization of Boundary Controlled Port-Hamiltonian Systems With Dynamic Feedback," \textit{IEEE Transactions on Automatic Control}, vol. 59, no. 10, pp. 2849-2855, 2014.

\bibitem{Raposo2004} C.A. Raposo, J. Ferreira, M.L. Santos and N.N.O. Castro, ``Exponential stability for the Timoshenko system with two weak damping," \textit{Applied Mathematics Letters}, vol. 18, 
 pp. 535–541, 2004.

\bibitem{Huang1985} F. Huang, ``Characteristic conditions for exponential stability of linear dynamical systems in Hilbert spaces",  \textit{Ann. Differential Equations}, vol. 1, pp. 43–55, 1985.

\bibitem{Shi2001} D. Shi and D. Feng, ``Exponential decay of Timoshenko beam with locally distributed feedback," \textit{IMA Journal of Mathematical Control and Information}, vol. 18, no. 3, pp. 395-403, 2001. 

\bibitem{Kim1987} J. U. Kim and Y. Renardy, ``Boundary Control of the Timoshenko Beam," \textit{SIAM Journal on Control and Optimization}, vol. 25, no. 6, pp. 1417-1429, 1987.

\bibitem{Almeida2013} D.S Almeida J\'unior, M.L. Santos and J.E Mu\~noz Rivera, ``Stability to weakly dissipative Timoshenko systems," \textit{Mathematical Methods in the Applied Sciences}, vol. 36, issue 14, pp. 1965-1976, 2013.

\bibitem{Haraux1988} A. Haraux and E. Zuazua, ``Decay Estimates for some Semilinear Damped Hyperbolic Problems," \textit{Archive for Rational Mechanics and Analysis}, vol. 100, pp. 191-206, 1988.

\bibitem{Zuazua1990} E. Zuazua, ``Exponential Decay for The Semilinear Wave Equation with Locally Distributed Damping," \textit{Communications in Partial Differential Equations}, vol. 15, issue 2, 1990.

\bibitem{Villegas2005} J.A. Villegas, H. Zwart, Y. Le Gorrec, B. Maschke and A.J. van der Schaft, ``Stability and stabilization of a class of boundary control systems" in \textit{Proc. of the 44th IEEE CDC}, pp. 3850-3855, 2005.

\bibitem{Augner2019} B. Augner, ``Well-Posedness and Stability of Infinite-Dimensional Linear Port-Hamiltonian Systems with Nonlinear Boundary Feedback," \textit{SIAM Journal on Control and Optimization}, vol. 57, no. 3, pp. 1818, 2019.

\bibitem{AugnerThesis} B. Augner, ``Stabilisation of Infinite-dimensional Port-Hamiltonian Systems via Dissipative Boundary Feedback," Ph.D. Dissertation, University of Wuppertal, 2016.

\bibitem{Villegas2007} J.A. Villegas, ``A port-Hamiltonian approach to distributed parameter systems," Ph.D. Dissertation, University of Twente, 2007.

\bibitem{Jacob2012} B. Jacob and H. Zwart, \textit{Linear port-Hamiltonian systems on Infinite-dimensional Spaces}, Number 223 in Operator Theory: Advances and Applications, Springer Verlag, 2012.

\bibitem{Macchelli2004} A. Macchelli and C. Melchiorri, ``Modeling and control of the Timoshenko beam: the distributed port Hamiltonian approach," \textit{SIAM Journal on Control and Optimization}, vol. 43, no. 2, pp. 743–767, 2004.

\bibitem{Hardy1959} G.H. Hardy, J. E. Littlewood and G. Polya, \textit{Inequalities} in Cambridge University Press, 1959, 2nd edition.

\bibitem{Curtain2020} R. Curtain and H. Zwart, \textit{Introduction to Infinite-Dimensional Linear Systems Theory, a state space approach}, Springer, 2020, 1st edition.

\bibitem{LeGorrec2005} Y. Le Gorrec, H. Zwart and B. Maschke, ``Dirac structures and boundary control systems associated with skew-symmetric differential operators," \textit{SIAM Journal on Control and Optimization}, vol. 44, pp. 1864-1892, 2005.

\bibitem{MattioniThesis}
A. Mattioni, ``Modelling and stability analysis of flexible robots: a distributed parameter port-Hamiltonian approach," Ph.D. Dissertation, Universit\'e de Bourgogne Franche-Comt\'e, 2021.

%\bibitem{Tian2017} X. Tian and Q. Zhang, ``Stability of a Timoshenko system with local Kelvin-Voigt damping," \textit{Zeitschrift für angewandte Mathematik und Physik}, vol. 68, issue 20, 2017.


%\bibitem{Messaoudi2008} S.A. Messaoudi, M. Pokojovy and B. Said-Houari, ``Nonlinear damped Timoshenko systems with second sound - Global existence and exponential stability," \textit{Mathematical Methods in the Applied Sciences}, vol. 32, pp. 505-534, 2008.


\end{thebibliography}

\end{document}